\documentclass{amsart}
\usepackage{geometry} % see geometry.pdf on how to lay out the page. There's lots.
\usepackage{amssymb}
\usepackage{faktor}
\usepackage[all,cmtip]{xy}
\usepackage{amsmath,amscd}
\usepackage{graphicx}
\usepackage{mathrsfs}
\usepackage{color}
\usepackage{bbm}
\usepackage[mathcal]{euscript}
\usepackage{hyperref}
\hypersetup{colorlinks=true,linkcolor=blue}

\newcommand{\End}{\mathcal{E}\mathrm{nd}}
\newcommand{\HHom}{\mathcal{H}\mathrm{om}}
\newcommand{\Ho}{\mathrm{Ho}}
\newcommand{\Map}{\mathrm{Map}}
\newcommand{\Hom}{\mathrm{Hom}}

\newcommand{\Emb}{\mathrm{Emb}}
\newcommand{\Conf}{\mathrm{Conf}}

\newcommand{\mcP}{\mathcal{P}}

\newcommand{\Fr}{\mathrm{Fr}}

\DeclareMathOperator*{\hocolim}{\mathrm{hocolim}}

\newcommand \vp{\varphi}

\newcommand\ob{\operatorname{Ob}}

\newcommand\op{\mathcal}
\newcommand\cat{\mathsf}

\newcommand\seq{\cat{Seq}}

\newcommand\adjunct[4]{\xymatrix{#1\ar @<1.25ex>[rr]^-{#3}&\perp&#2\ar @<1.25ex>[ll]^-{#4}}}
\newcommand\longadjunct[4]{\xymatrix{#1\ar @<1.25ex>[rrrr]^{#3}&&\perp&&#2\ar @<1.25ex>[llll]^{#4}}}

\newcommand\map{\operatorname{Hom}}

\newcommand\Mod[1]{\cat{Mod}_{#1}}

\newcommand\id{\mathrm{id}}

\renewcommand\S{\cat {sSet}}

\newcommand\A{\cat A}
\newcommand\F{\cat F}
\newcommand\C{\cat C}
\newcommand\D{\cat D}
\newcommand\M{\cat M}
\newcommand\E{\cat E}
\newcommand\V{\cat V}

\newcommand\Mfld{\cat{Mfld}}
\newcommand\Disk{\cat{Disk}}

\def\treeof(#1;#2){[#1;#2]}

\newcommand{\comp}{\relax}
\def\comp(#1;#2){#1\circ(#2)}

\theoremstyle{remark}\newtheorem*{question}{Question}
\theoremstyle{definition}\newtheorem{definition}{Definition}[section]
\theoremstyle{theorem}\newtheorem{lemma}[definition]{Lemma}

\theoremstyle{remark}\newtheorem{remark}[definition]{Remark}
\theoremstyle{definition}\newtheorem{notation}[definition]{Notation}
\theoremstyle{definition}
\theoremstyle{definition}
\theoremstyle{definition}\newtheorem{example}[definition]{Example}
\theoremstyle{theorem}\newtheorem{proposition}[definition]{Proposition}
\theoremstyle{theorem}\newtheorem{corollary}[definition]{Corollary}
\theoremstyle{theorem}\newtheorem{theorem}[definition]{Theorem}
\theoremstyle{definition}
\theoremstyle{theorem}\newtheorem{conjecture}[definition]{Conjecture}

\title{Configuration spaces of products}
\author{William Dwyer, Kathryn Hess, and Ben Knudsen}
\date{} % delete this line to display the current date

\begin{document}

\maketitle

\begin{abstract}
We show that the configuration spaces of a product of parallelizable manifolds may be recovered from those of the factors as the Boardman-Vogt tensor product of right modules over {the operads of little cubes of the appropriate dimension}. We also discuss an analogue of this result for manifolds that are not necessarily parallelizable, which involves a new operad of \emph{skew little cubes}.
\end{abstract}

%\setcounter{tocdepth}{1}
%\tableofcontents

\section{Introduction}

To a manifold $M$, there is associated a family of spaces of fundamental geometric and homotopical interest, the \emph{configuration spaces} \[\Conf_k(M):=\left\{(x_1,\ldots, x_k)\in M^k: x_i\neq x_j\text{ if }i\neq j\right\}.\] Since their introduction by Fadell-Neuwirth \cite{fadell-neuwirth}, these spaces have been the subject of intensive study from many different perspectives---see \cite{cohen-lada-may}, \cite{axelrod-singer}, and \cite{ghrist} for a small, but varied, sampling. Throughout the history of this investigation, a guiding theme has been the following.

\begin{question}
How do the configuration spaces of $M$ depend on $M$ itself?
\end{question}

\subsection{Invariance and decomposition}One interpretation of this question is to view the homotopy type of $\Conf_k(M)$ as an invariant of $M$. It is easy to see that this invariant is not a homotopy invariant---most of the configuration spaces of a point are empty, for example---but it is much less obvious that it is not even a homotopy invariant of compact manifolds of equal dimension; indeed, according to a theorem of Longoni-Salvatore \cite{longoni-salvatore}, the homotopy type of $\Conf_2$ distinguishes lens spaces that are homotopy equivalent but not homeomorphic. This lack of homotopy invariance is a subtle feature of the unstable homotopy type of configuration spaces, disappearing after sufficient suspension---see \cite{aouina-klein} and \cite{knudsen}.

A second interpretation of this question is as follows. Supposing that we are able to decompose $M$ in some fashion, is there a corresponding decomposition at the level of configuration spaces? For example, it is easy to see that \[\Conf_k(M\amalg N)\cong \coprod_{i+j=k}\Conf_i(M)\times\Conf_j(N)\times_{\Sigma_i\times\Sigma_j}\Sigma_k.\] This seeming triviality carries within it the seeds of deeper facts about the behavior of configuration spaces under collar gluings---see \cite{mcduff}, \cite{boedigheimer}, and \cite{ayala-francis}, for example.

In this note, we broaden this line of inquiry to encompass a different kind of decomposition. Specifically, we seek to characterize the configuration spaces of the product manifold $M\times N$ in terms of the configuration spaces of the factors. In order to address this problem, as with so many problems relating to configuration spaces, it will be profitable to invest in higher structures.

\subsection{Operads and additivity} The configuration space of $k$ points in $\mathbb{R}^m$ has the homotopy type of the space $\op C_m(k)$ of $k$ disjoint \emph{little $m$-cubes}, where a little cube is a rectilinear {self-embedding of $(-1,1)^m$}. Since the composite of two rectilinear embeddings is again rectilinear, this collection of spaces carries a rich algebraic structure, that of the \emph{little cubes operad} $\op C_m$ of Boardman-Vogt \cite{boardman-vogt:htpyeverything} and May \cite{may2}.

In fact, there is a whole menagerie of operads, the $E_m$-operads, with the configuration spaces of $\mathbb{R}^m$ as their underlying homotopy types---the $m$-dimensional little disks operad, Fulton-MacPherson operad, and Steiner operad, to name a few. Each of these models enjoys its own combination of features and drawbacks arising from the specifics of the geometry involved. 

For example, the little cubes operads in different dimensions may be related by the observation that the product of a little $m$-cube and a little $n$-cube is a little $(m+n)$-cube. In the language of operads, this geometric fact is expressed as the existence of a map of operads \[\iota \colon \op C_m\otimes\op C_n\to \op C_{m+n}\] where $\otimes$ denotes the \emph{Boardman-Vogt tensor product} \cite{boardman-vogt:lnm} codifying the concept of interchangeable operad structures. In fact, according to the ``additivity'' theorem of Dunn \cite{dunn} and Brinkmeier \cite{brinkmeier}, this map is a weak equivalence of operads. Thus, after passing to the operadic context, we recover the configuration spaces of the product manifold $\mathbb{R}^{m+n}=\mathbb{R}^m\times\mathbb{R}^n$ from those of the factors $\mathbb{R}^m$ and $\mathbb{R}^n$. 

\subsection{Modules and configuration spaces} From the operadic point of view, the homotopy types of the configuration spaces of a more general manifold $M$ are organized by a right $\op C_m$-module $\op C_M$, at least if $M$ is parallelizable. There is a notion of tensor product of operadic modules paralleling the Boardman-Vogt tensor product, which was constructed by the first two authors in \cite{dwyer-hess:bv}, and it is natural to ask: are the configuration spaces of a more general product manifold determined by those of its factors in terms of this lifted Boardman-Vogt tensor product of modules?

Our main result answers this question in the affirmative.

\begin{theorem}\label{thm:main}
Let $M$ be an $m$-manifold and $N$ an $n$-manifold. A choice of trivialization of each tangent bundle determines an isomorphism \[\Ho(\iota^*)(\op C_{M\times N})\cong \op C_M\otimes^\mathbb{L} \op C_N\] in the homotopy category of right $\op C_{m}\otimes \op C_{n}$-modules, where $\iota^*$ denotes the functor pulling back module structure along the operad map $\iota\colon \op C_m\otimes\op C_n\to \op C_{m+n}$ and $\otimes^\mathbb{L}$  the left derived Boardman-Vogt tensor product of modules.
\end{theorem}

This statement is somewhat imprecise; in fact, we require a choice of structure on the tangent bundle that can roughly be described as a trivialization up to dilation. See Theorem \ref{thm:fancy main} for a precise statement.

The broad strategy of the proof is to globalize the local equivalence supplied by the additivity theorem over the coordinate patches of a general manifold. We remark that part of the technical work required is verifying that the tensor product of right modules is homotopically well-behaved enough to admit a derived functor.

\subsection{Structured configuration spaces} In studying non-parellelizable manifolds, it is natural to consider instead, for a manifold $M$ equipped with a reduction to $G$ of the structure group of $TM$, the corresponding principal $G^k$-bundle $\Conf_k^G(M)\to \Conf_k(M)$, which we call the $k$th \emph{$G$-framed configuration space} of $M$. {We introduce a new family of operads, the \emph{skew little cubes} operads $\op C_m^G$,} and we conjecture that these too are additive in the sense that $\op C_m^G\otimes \op C_n^H$ and  $\op C_{m+n}^{G\times H}$ are weakly equivalent as operads. Assuming this conjecture, we prove a version of Theorem \ref{thm:main} asserting that the corresponding modules of structured configuration spaces are also additive---see Theorem \ref{thm:fancy main}.

\subsection{Future directions} Our work points to a number of avenues of research, which we hope to pursue in subsequent work.
\begin{itemize}
\item \emph{Spaces of embeddings}. From the point of view of embedding calculus \cite{debrito-weiss}, Theorem \ref{thm:main} implies a kind of ``change of rings'' equivalence at the level of Taylor approximations to the functor of embeddings into a fixed target, expressed in terms of the divided powers functor investigated by the first two authors \cite{dwyer-hess:bv}. What does this equivalence mean in terms of the geometry of embeddings of product manifolds?
\item \emph{Rational models}. With certain restrictions on $M$ in place, the rational homotopy types of the configuration spaces of $M$ as right modules admit explicit models \cite{idrissi, campos-willwacher}. How does the equivalence of Theorem \ref{thm:main} interact with these models? 
\item \emph{Equivariant additivity}. Does Conjecture \ref{conj:additivity} on the additivity of the skew little cubes operads hold?
\end{itemize}

\subsection{Relation to previous work} The ideas of Section \ref{section:multilocality} are drawn from the theory of factorization homology \cite{ayala-francis}, although we do not employ any of its formal machinery. The Boardman-Vogt tensor product of right modules is a special case of the tensor product of bimodules investigated by the first two authors \cite{dwyer-hess:bv}.

{Our results bear a family resemblance to those of} \cite{debrito-weiss:product} (see also \cite[5.4.5.5]{lurie2}); however, we know of no formal connection between them.

\section{Operadic reminders}

In this section we recall and fix notation for various operadic concepts used throughout this article, in particular the Boardman-Vogt tensor product of simplicial or topological operads. We assume that the reader already has a working knowledge of the theory of operads and their right modules. Relevant introductory references include \cite[I.5-7,\, II,\,III.11-12]{fresse}, \cite[5]{loday-vallette}, and \cite{markl-shnider-stasheff}.

\subsection{Operads and modules}\label{sec:operads and modules} Throughout this section, $(\V,\otimes, \map_{\V}, I)$ denotes a cocomplete, closed, symmetric monoidal category having an initial object. The examples of greatest interest are the Cartesian monoidal categories $\S$ of simplicial sets and $\cat{Top}$ of compactly generated weak Hausdorff spaces.

We write $\seq_\Sigma(\V):=\V^{\Sigma^{op}}$ for the category of symmetric sequences in $\V$, where $\Sigma$ denotes the groupoid of finite sets and bijections, and we employ the standard exponential notation $\V^{\cat C}$ for the category of functors from $\cat C$ to $\V$. A symmetric sequence $\op X$ is determined by a collection $\{\op X(k)\}_{k\geq0}$ of objects in $\V$ equipped with {right} symmetric group actions, where $\op X(k):=\op X(\{1,\ldots, k\})$ is the \emph{arity $k$ component} of $\op X$. Objects of $\V$ are naturally identified with symmetric sequences concentrated in arity 0, i.e., in which all higher arity entries are the initial object with its unique action.

When $\V$ is a model category, we shall say that a map of symmetric sequences is a weak equivalence or a fibration if it is so pointwise (note that this model structure differs from that of \cite{rezk}). In favorable circumstances---for example, when $\V$ is cofibrantly generated---these weak equivalences and fibrations determine a model structure on $\seq_\Sigma(\V)$ \cite[11.6.1]{hirschhorn}. In the presence of such a model structure, we refer to the cofibrations in this model structure as $\Sigma$-\emph{cofibrations}, and to the cofibrant objects as $\Sigma$-\emph{cofibrant.}

The category $\seq_\Sigma(\V)$ is naturally symmetric monoidal under the \emph{graded tensor product}, specified by \[(\op X\odot \op Y)(k)=\coprod_{i+j=k}\op X(i)\otimes \op Y(j)\otimes_{\Sigma_i\times\Sigma_j}\Sigma_k.\] In addition, there is the \emph{composition product}, which may defined in terms of the graded tensor product by the formula \[\op X\circ \op Y=\coprod_{\ell\geq0} \op X(\ell)\odot_{\Sigma_\ell}\op Y^{\odot \ell}.\] The composition product determines a second (non-symmetric) monoidal structure on $\seq_\Sigma(\V)$ for which the unit $\mathcal J$ is the symmetric sequence that is the unit in $\V$ in arity 1 and the initial object in $\V$ in all other arities.

\begin{definition}
An \emph{operad} in $\V$ is a monoid in $(\seq_\Sigma(\V),\circ, \mathcal J)$.
\end{definition}

A map of operads is simply a map of monoids, and we obtain in this way a category $\cat{Op}(\V)$. When $\V$ is a model category, we say that a map of operads is a weak equivalence if the underlying map of symmetric sequences is so. In favorable circumstances---for example, when $\V=\S$ or $\V=\cat{Top}$---these weak equivalences are the weak equivalences of a model structure on $\cat{Op}(\V)$ (see \cite{berger-moerdijk}).

\begin{example}\label{example:endomorphism}
An object $V\in \V$ determines a canonical operad $\End_\V(V)$ in $\V$, the \emph{endomorphism operad} of $V$. As a symmetric sequence, \[\End_\V(V)(k):=\Hom_\V(V^{\otimes k}, V),\] and the components of the monoid structure map are given by composition in $\V$.
\end{example}

For an operad $\mcP$, we write $\Mod{\mcP}(\V)$ for the $\V$-category of right $\mcP$-modules (here, and throughout, a $\V$-\emph{category} is a $\V$-enriched category, and we likewise use the standard terminology $\V$-functor, $\V$-natural transformation, and so on). The free right $\mcP$-module on the symmetric sequence $\op X$ is $\op X\circ\mcP$. Since we have no cause to consider other types of module structure in this work, we shall often refer simply to $\mcP$-modules.

\subsection{Modules as presheaves} \label{section:modules as presheaves}
We now recall the standard correspondence between right $\op P$-modules and presheaves on a certain category associated to $\op P$---see \cite[3]{arone-turchin}, for example. Denoting by $\F$ the category of finite sets, we write $\F(\op P)$ for the $\V$-category with objects the objects of $\F$ and \[\Hom_{\F(\op P)}(X,Y)= \coprod_{f:X\to Y} \bigotimes_{y\in Y}\op P \big( f^{-1}(y)\big),\] where the coproduct is indexed on the set $\Hom_\F(X,Y)$. Composition is defined in the obvious manner using the operad structure of $\mcP$. 

\begin{lemma}[{\cite[3.3]{arone-turchin}}]\label{lem:right-mod} The category $\Mod{\op P}(\V)$ of right $\op P$-modules in $\V$ is equivalent as a $\V$-category to $\V^{\F(\op P)}$, the $\V$-category of $\V$-functors and $\V$-natural transformations from $\F(\op P)$ to $\V$ (cf. section \ref{sec:ext-prod}).
\end{lemma}

Henceforth, abusing notation slightly, we identify $\Mod{\op P}(\V)$ and $\V^{\F(\op P)}$. 

\begin{example}\label{example:yoneda}
For any object $V\in \op V$, there is a canonical $\V$-functor $\F(\End_\V(V))\to \V$ sending $I$ to $V^{\otimes I}$. Restricting the $\V$-enriched Yoneda embedding along this functor, we obtain the $\V$-functor \[\HHom_\V(V,-):\V\to \Mod{\End_\V(V)}(\V),\] which, when evaluated on an object $W$, returns the module given in arity $k$ by $\Hom_\V(V^{\otimes k}, W)$, with structure maps given by composition in $\V$. Note that $\HHom_\V(V,-)(1)=\Hom_\V(V,-)$.
\end{example}

The construction $\F(\mcP)$ is functorial in an obvious way, so a map $\varphi:\mcP\to \op Q$ of operads gives rise to a \emph{base change adjunction} \[\adjunct{\Mod{\op P}(\V)}{\Mod{\op Q}(\V)}{\vp_{!}}{\vp^{*}}\] by restriction and left Kan extension along $\F(\varphi)$. As we will see in section \ref{section:htpy}, under mild cofibrancy conditions, base change is homotopically meaningful.

\begin{remark}\label{remark:free-mod}  The unit map of an operad $\op P$ gives rise to a $\V$-functor $\zeta_{\op P}: \Sigma ^{op} \to F(\op P)$ that is identity on objects, which in turn induces an adjuction
$$\adjunct {\seq(\V)}{\Mod{\op P}(\V)} {(\zeta_{\op P})_!} {\zeta_{\op P}^{*}}.$$
The left adjoint $(\zeta_{\op P})_!$ is precisely the free $\op P$-module functor from the presheaf perspective.
\end{remark}

\subsection{Simplicial and topological operads} 
We now establish some useful notation in the cases $\V= \S$ and $\V= \cat {Top}$. In view of the formula \[(\op X\circ \op Y)(k)=\coprod_{\sum_{i=1}^\ell k_i=k} \op X(\ell)\times_{\Sigma_\ell}\bigg(\prod_{i=1}^\ell \op Y(k_i)\times_{\prod_{i=1}^\ell\
\Sigma_i}\Sigma_k\bigg),\] we may represent a typical element $(\op X\circ \op Y)(k)$ by a tuple $(x; y_{1},...,y_{l};\tau)$, where $x\in \op X(\ell)$, $y_{i}\in \op Y(k_{i})$, and $\tau \in \Sigma_{k}$, subject to the relations 
\[(x; y_{1},...,y_{\ell};\tau _{1}\oplus \cdots \oplus \tau_{\ell})\sim (x; y_{1}\cdot \tau_{1}^{-1}, \cdots, y_{\ell}\cdot \tau_{\ell}^{-1};\id)\]
and
\[(x\cdot \sigma^{-1}; y_{1},...,y_{\ell};\id)\sim (x; y_{\sigma(1)},...,y_{\sigma(\ell)};\id)\]
for all $\sigma \in \Sigma_\ell$ and $\tau_{i}\in \Sigma_{k_{i}}$, where $\tau _{1}\oplus \cdots \oplus \tau _{\ell}\in\Sigma_{k}$ is the block permutation specified by
\[(\tau_{1}\oplus \cdots \oplus \tau _{\ell})(r)=\tau_{j}\big(r-\sum_{i=1}^{j-1}k_{i}\big) + \sum_{i=1}^{j-1}k_{i}\quad\text{for all}\quad  \sum_{i=1}^{j-1}k_{i}<r\leq  \sum_{i=1}^{j}k_{i}.\] In this representation, the right action of $\Sigma_{k}$ on $\op X\circ \op Y$ is given by 
\[(x; y_{1},...,y_{\ell};\id)\cdot \tau=(x; y_{1},...,y_{\ell};\tau).\]

If $\op P$ is an operad with multiplication map $\mu:\op P\circ \op P\to \op P$ and $p\in \op P(k)$, then we write 
\[p(p_{1},...,p_{k}):=\mu (p;p_{1},...,p_{k};\id)\in \op P(n)\]
for any $p_{i}\in \op P(k_{i})$ with $\sum_{i=1}^\ell k_{i}=k$.
Note that, since $\mu$ is equivariant, it is specified by its values on elements of $\op P\circ \op P$ with representatives of this form.

\subsection{The Boardman-Vogt tensor product}\label{sec:bv}
We consider only simplicial and topological operads in this section, writing $\cat {Op}$ for $\cat {Op}(\S)$ or $\cat{Op}(\cat {Top})$.

The Boardman-Vogt tensor product \cite{boardman-vogt:lnm}
$$-\otimes -: \cat {Op} \times \cat {Op} \to \cat {Op},$$
which endows the category $\cat {Op}$ with a symmetric monoidal structure, codifies interchanging algebraic structures.  For  all $\op P, \op Q\in \cat {Op}$, a $(\op P\otimes \op Q)$-algebra can be viewed as a $\op P$-algebra in the category of $\op Q$-algebras or as a $\op Q$-algebra in the category of $\op P$-algebras.

\begin{definition}\label{definition:bv-op}\cite{boardman-vogt:lnm}, \cite{dunn} The \emph{Boardman-Vogt tensor product} of operads $\op P$ and $\op Q$ is the operad $\op P\otimes \op Q$ that is the quotient of the coproduct $\op P\coprod \op Q$ of operads by the equivalence relation generated by
$$(p;\underbrace{q,..,q}_{k};\id) \sim (q;\underbrace{p,...,p}_{l};\tau_{k,l})$$
for all $p\in \op P(k)$ and $q\in \op Q(l)$, where $\tau_{k,l}\in \Sigma_{kl}$ is the transpose permutation that ``exchanges rows and columns'', i.e., for all $1\leq m=(i-1)l+j\leq kl$, where $1\leq i\leq k$ and $1\leq j\leq l$, 
$$\tau_{k,l}(m)=(j-1)k+i.$$
\end{definition}

\begin{notation}\label{notn:tensor} We let $p\otimes q$ denote the common equivalence class of  
$$(p;\underbrace{q,..,q}_{k};\id)\quad\text{and}\quad(q;\underbrace{p,...,p}_{l};\tau_{k,l})$$ in $(\op P\otimes \op Q)(kl)$
for all $p\in \op P(k)$ and $q\in \op Q(l)$.
\end{notation}

\begin{remark}\label{rmk:switch} In terms of the notation above, if $p\in \op P(k)$ and $q\in \op Q(l)$, then 
$$p\otimes q = q\otimes p \cdot \tau_{k,l}.$$
\end{remark}

\section{The lifted Boardman-Vogt tensor product}\label{section:bv}

In this section, we introduce a version of the Boardman-Vogt tensor product at the level of right modules. A more general tensor product, valid for operadic bimodules, is defined in \cite{dwyer-hess:bv}, but a simplified construction of this operation is available when the operads acting on the left are all trivial. This simplified approach is convenient for homotopy theoretic applications; in particular, it allows for an easy proof of the key result that tensoring with a cofibrant module is a left Quillen functor, which appears below as Proposition \ref{prop:tensor with cofibrant}.

Throughout this section, $(\V,\otimes, \map_{\V}, I)$ denotes a closed, symmetric monoidal category, and the hom objects in a $\V$-category $\E$ are denoted $\map_{\E}(-,-)$.

\subsection{External products for enriched categories}\label{sec:ext-prod}

Let $\A$ and $\C$ be $\V$-categories, and assume that $\C$ is $\V$-equivalent to a small $\V$-category. As explained in \cite[2.1-2]{kelly}, there is a $\V$-category $\A^{\C}$ with objects $\V$-functors from $\C$ and morphism objects given by the end formula
$$\map_{\A^{\C}}(F,F')=\int_{C\in \ob \C}\map_{\A}(FC, F'C).$$

Constructions of the following sort come up naturally in our work on operads.

\begin{definition}\label{definition:div-power} Let $F: \C\to \V$ and $F':\D \to \V$ be $\V$-functors. The \emph{external tensor product} of $F$ and $F'$ is the functor
$$F\boxtimes F': \C \times \D \to \V$$
defined on objects by $(F\boxtimes F')(c,d)= F(c) \otimes F'(d)$ and similarly on hom objects. 

Let $H:\C \times \D \to \V$ be a $\V$-functor, and let $d\in \ob \D$.  The \emph{$d$-divided power of $H$} is the $\V$-functor
$$\gamma^{\C}_{d}(H): \C \to \V$$
that is defined on objects by $\gamma^{\C}_{d}(H)(c)=H(c,d)$ and similarly on hom objects.  The \emph{$c$-divided power of $H$} for any $c\in \ob \C$,
$$\gamma^{\D}_{c}(H): \D \to \V,$$
is defined analogously.
\end{definition}

The proof of the lemma below is an easy  exercise in enriched category theory, using the description of the $\V$-enrichment of functor categories given above.

\begin{lemma}\label{lem:funcat-adjoin}  Let $\V$ be a closed, symmetric monoidal category, and let $\C$ and $\D$ be small $\V$-categories. For any $\V$-functor $F': \D \to \V$, there is an adjunction
$$\adjunct{ \V^{\C}}{\V ^{\C \times \D},}{-\boxtimes F'}{\widetilde{\map}_{\D}(F',-)}$$
where the right adjoint is specified on $H\in \ob \V ^{\C \times \D}$ by
$$\widetilde{\map}_{\D}(F',H)(c)=\map_{\V^{\D}}\big(F', \gamma^{\D}_{c}(H)\big).$$
\end{lemma}

If $\V $ is cocomplete, a $\V$-functor $\Phi:\C\times\D \to \E$ between small $\V$-categories induces an adjunction
$$\adjunct{\V^{\C\times \D}}{\V^{\E},}{\Phi_!}{\Phi^{*}}$$
where $\Phi_!$ denotes the enriched left Kan extension along $\Phi$ \cite[4.50]{kelly}. We will be interested in particular examples of the resulting composite adjunction
$$\longadjunct{ \V^{\C}}{\V ^{\E}.}{\Phi_!\circ(-\boxtimes F')}{\widetilde{\map}_{\D}\big(F',\Phi^{*}(-)\big)}$$ The motivating example for this construction, which arose in \cite{dwyer-hess:bv}, is formulated as follows.

\begin{example}\label{ex:matrix-mon} Let $\nu:\Sigma \times \Sigma \to \Sigma$ be the functor specified by $\nu(I,J)=I\times J$ for all $I,J\in \ob \Sigma $, while
$\nu_{I,J}:\Sigma_{I}\times \Sigma_{J}\to \Sigma_{I\times J}$
is the homomorphism given by
$$\nu_{I,J}(\sigma, \tau)(i,j)=\big(\sigma(i), \tau (j)\big).$$ The \emph{matrix tensor product} on $\seq_{\Sigma}(\V)$ \cite[1.3]{dwyer-hess:bv}, denoted $\square$, is the composite
$$-\square-:\seq_{\Sigma}(\V)\times \seq_{\Sigma}(\V) \xrightarrow {-\boxtimes-} \V^{\Sigma^{op}\times \Sigma^{op} }\xrightarrow {(\nu^{op})_!}\seq_{\Sigma}(\V).$$ Explicitly, this tensor product is given by the formula
$$(\op X \square \op Y)(k)=\coprod_{ij=k}\big(\op X(i)\times \op Y(j)\big)\times _{\Sigma_{i}\times \Sigma _{j}} \Sigma_{k}$$
for symmetric sequences $\op X$ and $\op Y$. The unit for the matrix monoidal structure is the symmetric sequence $\op I$ that is the unit $I$  in arity 1 and the initial object in all other arities.

\end{example}

\subsection{The homotopy theory of external products}\label{section:htpy}
Let $\M$ be a monoidal model category that is cofibrantly generated in the enriched sense \cite[13.4]{riehl}, and let $\varnothing$ denote its initial object.  By \cite[13.4.5, 13.5.2]{riehl}, examples of such monoidal model categories include the category $\S$ of simplicial sets with Cartesian product and the Kan-Quillen model structure, the category $\cat{Top}$ with the $k$-ified Cartesian monoidal structure and the Serre model structure, and the category of chain complexes over a commutative ring, equipped with its Hurewicz-type model structure (where the weak equivalences are chain homotopy equivalences) and the usual tensor product. 

For any $\M$-enriched category $\C$, let $\C_{\delta}$ denote the category with the same objects as $\C$ and with 
$$\C_{\delta}(C,C')=\begin{cases} I&: C=C'\\ \varnothing&: C\not=C',\end{cases}$$ 
and let $\iota_{\C}:\C_{\delta}\to \C$ denote the obvious functor.  

\begin{theorem}[{\cite[13.5.1]{riehl}}]\label{thm:diag-model-cat}  For any small $\M$-enriched category $\C$, there is a cofibrantly generated, $\M$-model category structure on $\M^{\C}$ that is right-induced by the $\M$-adjunction
$$\adjunct{\M^{\C_{\delta}}}{\M^{\C}}{(\iota_{\C})_!}{\iota_{\C}^{*}},$$
i.e., the fibrations and weak equivalences in $\M^{\C}$ are defined objectwise.
\end{theorem}

The model structure of Theorem \ref{thm:diag-model-cat} is usually called the \emph{projective model structure} on $\M^{\C}$.   For any operad $\op P$ in $\M$, applying Theorem \ref{thm:diag-model-cat} to the category $\Mod{\op P}(\M)$, in the guise of the presheaf category $\M^{\F(\op P)}$, gives rise to a projective model structure on the category of right $\op P$-modules.

The following easy lemma asserts that the projective model structure behaves well with respect to $\M$-functors in the source.

\begin{lemma}\label{lem:module hocolim}  For any $\M$-functor $\Phi: \C\to \D$, the induced adjunction $$\longadjunct { \M^{\C}}{\M ^{\D}}{\Phi_!}{\Phi^{*}}$$ is a Quillen pair, and $\Ho(\Phi^*)$ preserves homotopy colimits.
\end{lemma}
\begin{proof}
The first statement is an immediate consequence of the fact that precomposition preserves objectwise fibrations and weak equivalences. The second follows from the fact that $\Phi^{*}$ preserves all weak equivalences. It follows that the induced functor on homotopy categories $\Ho(\Phi^{*})\colon \Ho(\M ^{\D}) \to \Ho (\M ^{\C})$ is both a total left and a total right derived functor, which implies in particular that $\Ho(\Phi^{*})$ commutes with homotopy colimits \cite[2.10]{werndli}.
\end{proof}

In particular, if $\varphi: \op P \to \op Q$ is a map of operads in $\M$, then the base change adjunction associated to $\varphi$ is a Quillen pair with respect to the projective model structures.  When $\vp$ is a weak equivalence, we can say even more.

\begin{proposition}[{\cite[15.B]{fresse}}] If $\varphi:\op P\to \op Q$ is a {weak equivalence} of operads in $\M$ such that $\op P$ and $\op Q$ are aritywise cofibrant, then the associated base change adjunction \[{\adjunct{\Mod{\op P}(\M)}{\Mod{\op Q}(\M)}{\vp_{!}}{\vp^{*}}}\] is a Quillen equivalence.
\end{proposition}

Adjunctions of the types considered in section \ref{sec:ext-prod} become homotopically meaningful with respect to the projective model structure.

\begin{proposition}\label{prop:quillen-pair} Let $\C$ and $\D$ be small $\M$-enriched categories. For any cofibrant object $F'$ in $\M^{\D}$, the adjunction 
$$\adjunct { \M^{\C}}{\M ^{\C \times \D}}{-\boxtimes F'}{\widetilde{\map}_{\D}(F',-)}$$
of Lemma \ref{lem:funcat-adjoin} is a Quillen pair with respect to the projective model structures.
\end{proposition}
\begin{proof} Since fibrations and weak equivalences in $\M^{\C}$ are defined objectwise, the $c$-divided power functor $\gamma^{\D}_{c}: \M^{\C\times \D}\to \M^{\D}$ preserves fibrations and weak equivalences for every $c\in \ob\C$. Thus, because $F'$ is cofibrant, and $\M^{\D}$ is an $\M$-model category, the functor $\map_{\M^{\D}}\big(F', \gamma^{\D}_{c}(-)\big)$ also preserves fibrations and weak equivalences.
\end{proof}

\begin{corollary}\label{cor:pairing}  Let $\Phi: \C \times \D \to \E$ be a functor, where $\C$, $\D$, and $\E$ are small categories, and let $\M$ be a monoidal model category.  For any cofibrant object $F'$ in $\M^{\D}$, the adjunction 
$$\longadjunct { \M^{\C}}{\M ^{\E}}{\Phi_!\circ(-\boxtimes F')}{\widetilde{\map}_{\D}\big(F',\Phi^{*}(-) \big)}$$
is a Quillen pair with respect to the projective model category structures.
\end{corollary}

\begin{proof}  It is obvious that 
$$\adjunct{ \M^{\C\times \D}}{\M ^{\E}}{\Phi_!}{\Phi^{*} }$$
is also a Quillen pair with respect to the projective model category structures, since precomposition with $\Phi$ preserves objectwise weak equivalences and fibrations.
\end{proof}

\begin{example} Corollary \ref{cor:pairing} implies that for all cofibrant  $\op Y\in \seq_{\Sigma}(\M)$, the adjunction
$$\longadjunct{\seq_{\Sigma}(\M)}{\seq_{\Sigma}(\M)}{-\square \op Y}{\widetilde{\map}_{\Sigma^{op}}\big(\op Y,\nu^{*}(-) \big)}$$
is a Quillen pair. This adjunction appeared in \cite[1.12]{dwyer-hess:bv}, using simplified notation $\gamma^{\Sigma}_{\bullet}=\gamma^{\Sigma^{op}}_{\bullet}(\nu^{*}(-))$, but without any reference to model structures.
\end{example}

\subsection{Application to operadic modules}\label{sec:opmod} 

We restrict henceforth to the case where $\M=\S$ or $\M=\cat {Top}$, with their Cartesian monoidal structures. Recall that, if $\op P$ is an operad in $\M$ and $X$ and $Y$ are finite sets, then an element of $\map_{\F(\op P)}(X,Y)$ is a pair $\big(f, (p_{j})_{y\in Y}\big)$, where $f: X\to Y$ is a map and $p_{y}\in \op P\big(f^{-1}(y)\big)$.

If $\op Q$ is a second operad, there is an $\M$-functor
$$\mu: \F(\op P) \times \F (\op Q) \to \F (\op P\otimes \op Q),$$
extending the functor $\nu: \Sigma \times \Sigma \to \Sigma$ of Example \ref{ex:matrix-mon}, defined on objects by $\mu(X,X')= X\times X'$ and on hom objects by declaring that the map
$$\map_{\F(\op P)}(X,Y)\times \map_{\F(\op Q)}(X',Y') \to \map_{\F(\op P\otimes \op Q)}(X\times X', Y\times Y')$$
is specified by
$$\big( \big(f, (p_{y})_{y\in Y}\big), \big(g, (q_{y'})_{y'\in Y'}\big)\Big)\mapsto \big( f\times g, (p_{y}\otimes q_{y'})_{(y,y')\in Y\times Y'}\big).$$

\begin{definition}\label{defn:bv-module} Let $\op P$ and $\op Q$ be operads in $\M$, $F$ a $\op P$-module, and $F'$ a $\op Q$-module.  The \emph{(lifted) Boardman-Vogt tensor product} of $F$ and $F'$ is the $\op P \otimes \op Q$-module
$$F\otimes F':= \mu_! (F\boxtimes F').$$ 
\end{definition}

We note the following comparison, although we do not use it.

\begin{proposition}
The Boardman-Vogt tensor product of right modules coincides with that defined in \cite{dwyer-hess:bv}.
\end{proposition}
\begin{proof}
It suffices to check that the two definitions agree on free modules, since every module is a coequalizer of free modules, and both constructions preserve colimits in each variable. We therefore need to show that for any symmetric sequences $\op X$ and $\op Y$ and any operads $\op P$ and $\op Q$,
$$(\zeta_{\op P})_!(\op X) \otimes (\zeta_{\op Q})_!(\op Y)\cong (\zeta_{\op P\otimes \op Q})_!(\op X\square \op Y).$$
(See Remark \ref{remark:free-mod} and Example \ref{ex:matrix-mon} for reminders of the notation used here, and see Theorem 1.14 in \cite{dwyer-hess:bv} for why this is the criterion to be checked.)

The desired isomorphism is a straightforward consequence of the facts that 
$$\zeta_{\op P\otimes \op Q}\circ \nu = \mu \circ (\zeta_{\op P}\times \zeta_{\op Q}),$$
which follows immediately from the definitions, and of the natural isomorphism in $\M$
$$\Hom_{\M^{\F(\op P)\times \F (\op Q)}}\left((\zeta_{\op P})_!(\op X) \boxtimes (\zeta_{\op Q})_!(\op Y), F'\right) \cong \Hom_{\M^{\Sigma^{op}\times \Sigma^{op}}}\big(\op X \boxtimes \op Y, (\zeta_{\op P}\times \zeta_{\op Q})^{*}F'\big),$$
which follows easily from a string of enriched adjunctions. 
\end{proof}

The next result is an immediate consequence of Corollary \ref{cor:pairing}.

\begin {proposition}\label{prop:tensor with cofibrant} Let $\op P$ and $\op Q$ be operads in $\M$, which is  $\S$ or $\cat {Top}$, and let $F'$ be a right $\op Q$-module.  If $F'$ is cofibrant in the projective model structure, then 
$$\longadjunct{\Mod{\op P}(\M)}{\Mod{\op P\otimes \op Q}(\M)}{-\otimes F'}{\widetilde{\map}_{\F(\op Q)}\big(F',\mu^{*}(-) \big)}$$
is a Quillen pair.
\end{proposition}

It follows that the functor $\Mod{\op P}(\M)\times\Mod{\op Q}(\M)\to \Mod{\op P\otimes\op Q}(\M)$ preserves weak equivalences between cofibrant objects, so it admits a total left derived functor \[-\otimes^\mathbb{L}-:\Ho(\Mod{\op P}(\M))\times\Ho(\Mod{\op Q}(\M))\to \Ho(\Mod{\op P\otimes\op Q}(\M)),\] the \emph{derived (lifted) Boardman-Vogt tensor product}. Explicitly, the derived tensor product may be computed by choosing cofibrant representative modules, forming the Boardman-Vogt tensor product, and passing to the homotopy category.

\section{Structured configuration spaces and skew cubes}

We turn now to the geometric side of our study, introducing and studying the operads and modules that organize the homotopy types of the configuration spaces of interest. 

\subsection{Structured manifolds} We will be interested in manifolds equipped with tangential structures. We adopt the categorical approach of \cite[V.5-10]{andrade}, whose point-set model for the space of structured embeddings provides a topologically enriched category with a strict composition (see \cite{ayala-francis} for an $\infty$-categorical approach to the same ideas). 

\begin{definition}
Let $M$ be an $m$-manifold and $G\to GL(m)$ a continuous homomorphism, and write $\Fr_{M}$ for the frame bundle of the tangent bundle of $M$. A \emph{$G$-framing} on $M$ is a principal $G$-bundle $\Fr^G_M$ together with an isomorphism \[\varphi_M:\Fr^G_M\times_G GL(m)\xrightarrow{\cong} \Fr_M\] of principal $GL(m)$-bundles covering the identity. A \emph{framing} is a $G$-framing with $G$ trivial.
\end{definition}

Abusively, our notation will never reflect the fact that the notion of a $G$-framing depends on the homomorphism $G\to GL(m)$, and we typically abbreviate to $M$ the triple constituting a $G$-framed manifold.

\begin{example}\label{example:euclidean space}
The manifold $\mathbb{R}^m$ is canonically $G$-framed for any $G\to GL(m)$.
\end{example}

When we write $\mathbb{R}^m$ in what follows, we implicitly refer to this standard $G$-structure.

\begin{example}\label{example:open and disjoint union}
Open submanifolds and disjoint unions of $G$-framed manifolds are canonically $G$-framed.
\end{example}

\begin{example}\label{example:cartesian product}
The Cartesian product of a $G$-framed manifold and an $H$-framed manifold is canonically $G\times H$-framed.
\end{example}

Recall that an embedding $f$ of manifolds induces a bundle map $Df$ on the corresponding frame bundles.

\begin{definition}
Let $M_1$ and $M_2$ be $G$-framed manifolds. A $G$-\emph{framed embedding} consists of an embedding $f:M_1\to M_2$, a bundle map $\tilde f:\Fr_{M_1}^G\to \Fr_{M_2}^G$, and a $GL(m)$-equivariant homotopy $h:\Fr_{M_1}\times[0,1]\to \Fr_{M_2}$ from $Df$ to  the composite $\varphi_{M_2}\circ (\tilde f\times_GGL(m))\circ\varphi_{M_1}^{-1}$. We require that $\tilde f$ and $h$ each cover $f$.
\end{definition}

The set $\Emb^G(M_1, M_2)$ of $G$-framed embeddings is naturally a subspace of the standard model for the homotopy pullback $P$ depicted in the following commuting diagram: \[\xymatrix{\Emb^G(M_1, M_2)\ar[dr]^-\sim\ar@/^1.4pc/[drr]\ar@/_1.4pc/[ddr]\\
&P\ar[r]\ar[d]& \Map^G(\Fr_{M_1}^G, \Fr_{M_2}^G)\ar[d]\\
&\Emb(M_1, M_2)\ar[r]^-D&\Map^{GL(m)}(\Fr_{M_1}, \Fr_{M_2}).
}\] As shown in \cite[V.9.1-2]{andrade}, the inclusion is a weak equivalence as indicated. The advantage of working with this model for the homotopy pullback is that one may compose $G$-framed embeddings using composition of embeddings and bundle maps and \emph{pointwise} composition of homotopies. We obtain in this way a topological category $\Mfld^G_m$, which we regard as symmetric monoidal under disjoint union. 

Note that, by combining Examples \ref{example:open and disjoint union} and \ref{example:cartesian product}, the configuration sapce $\Conf_k(M)\subseteq M^k$ is canonically $G^k$-framed whenever $M$ is $G$-framed, so it is sensible to make the following definition.

\begin{definition}
The $G$-\emph{framed configuration space} of $k$ points in $M$ is the $G^k\rtimes\Sigma_k$-space \[\Conf^G_k(M):=\Fr_{\Conf_k(M)}^{G^k},\] where $\Sigma_k$ acts on $G^k$ by permuting the factors.
\end{definition}

\begin{proposition}[{\cite[14.4]{andrade}}]\label{prop:embedding and conf}
Let $M$ be a $G$-framed manifold. For each $k\geq0$, the map \[\Emb^G(\amalg_k\mathbb{R}^m, M)\longrightarrow\Conf^G_k(M)\] induced by evaluation at the origin is a $G^k\rtimes \Sigma_k$-equivariant weak equivalence.
\end{proposition}

The obvious analogues of Examples \ref{example:open and disjoint union} and \ref{example:cartesian product} apply to $G$-framed embeddings, so the $G$-framed configuration space extends canonically to a functor \[\Conf_k^G:\Mfld_m^G\to \cat{Top}.\] In fact, this functor is also an enriched functor in a natural way, but we will not use this fact.

\begin{remark}
Setting $M=\mathbb{R}^m$ and $k=1$ in Proposition \ref{prop:embedding and conf}, it follows that the topological monoid of $G$-framed self-embeddings of $\mathbb{R}^m$---considered, as always, with its standard $G$-structure---is weakly equivalent to $G$. In particular, the full topological subgroupoid of $\Mfld^G_m$ containing the single object $\mathbb{R}^m$ is equivalent to $BG$.
\end{remark}

\subsection{Skew little cubes}\label{section:additivity} 

We denote by $\Lambda(m)\subseteq GL(m)$ the subgroup of diagonal matrices with positive entries. At times, we may refer to $\Lambda(m)$ as the \emph{dilation group}. Recall that the QR decomposition in linear algebra provides a canonical identification of the right coset space $O(m)\backslash GL(m)$ with the space of upper triangular matrices with positive diagonal entries. 

\begin{definition}\label{def:dilation rep}
A \emph{dilation representation} is a continuous group homomorphism $\rho:G\to GL(m)$ such that $\mathrm{im}(\rho)=(\mathrm{im}(\rho)\cap O(m))\cdot \Lambda(m).$
%\begin{enumerate} 
%\item $\rho(G)$ contains $\Lambda(m)$, and
%\item the composite $G\to GL(m)\to  GL(m)/O(m)$ factors through the inclusion of $\Lambda(m)$. \bltext{[K.: I realize now that I find this formulation a bit confusing. Do you mean that the image of the composite is contained in the image of $\Lambda (m)\to GL(m)\to  GL(m)/O(m)$? This would mean that actually the two images were the same, by condition (1). ]}
%\end{enumerate}
\end{definition}

Thus, the image of $\rho$ contains the entire dilation group $\Lambda(m)$, and the upper triangular matrix appearing in the QR factorization of any element in this image lies in $\Lambda(m)$.

%The images of such representations may be characterized.

%\begin{lemma}\label{lem:dilation image}
%If $\rho:G\to GL(m)$ is a dilation representation, then \[\mathrm{im}(\rho)=(\mathrm{im}(\rho)\cap O(m))\cdot \Lambda(m).\]
%\end{lemma}
%\begin{proof}
%Condition (1) of Definition \ref{def:dilation rep} implies the inclusion $(\mathrm{im}(\rho)\cap O(m))\cdot \Lambda(m)\subseteq \mathrm{im}(\rho)$, and condition (2) implies the reverse inclusion. 
%\end{proof}

{\begin{example}
An example of a dilation representation that is not an inclusion is provided by the universal cover of $SO(m)\cdot \Lambda(m)$, which is homotopy equivalent to the spin group.
\end{example}}

For the remainder of this section, we fix a dilation representation $\rho:G\to GL(m)$, which we abusively leave implicit in the notation. {We write $\Box^m:=(-1,1)^m\subseteq \mathbb{R}^m$ for the open $m$-cube of side-length 2 centered at the origin.}

\begin{definition}
A $G$-\emph{skew little cube} is a pair $(v,g)$ with $v\in\Box^m$ and $g\in G$ such that the formula $f_{v,g}(x)=\rho(g)x+v$ specifies an embedding $f_{v,g}:\Box^m\to \Box^m$. A \emph{little $m$-cube} is a $\Lambda(m)$-skew little cube.
\end{definition}

We identify a skew little cube $(v,g)$ with the associated embedding $f_{v,g}$. We write $\op {C}^G_m(k)$
for the set of $k$-tuples of $G$-skew little cubes with pairwise disjoint images (cf. Figure 1)\begin{figure}[h]\label{figure}
 \centerline{\includegraphics[scale=0.3]{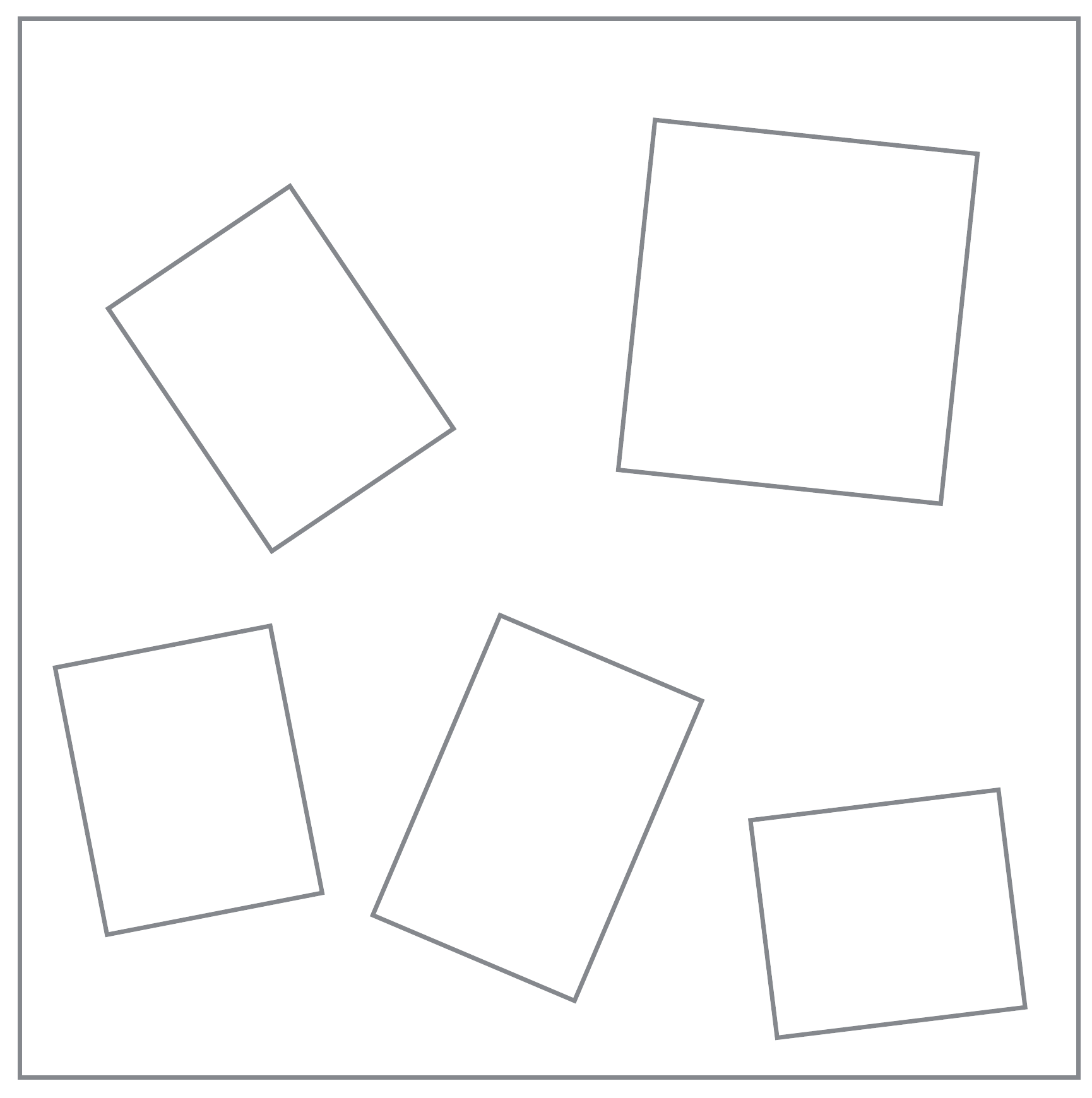}}
 \caption{An element of $\op C^{O(2)\cdot \Lambda (2)}_{2}(5)$}
\end{figure} and topologize this set as a subspace of ${(-1,1)}^{2mk}\times G^k$ by recording the images of ${\pm1/2}$ in each coordinate, or, alternatively, as a subspace of $\Map(\amalg_k\Box^m,\Box^m)\times G^k$---as in \cite[4.2]{may2}, these topologies coincide. 
It is easy to see that $G$-skew little cubes are closed under composition, so $\op {C}^G_m$ forms an operad. {Observe that $\op {C}^{\Lambda(m)}_m$ is isomorphic to the usual little $m$-cubes operad $\op C_m$.}

\begin{remark}
The definition of $\op C_m^G$ is sensible for any homomorphism $\rho:G\to GL(m)$; however, unless $\rho$ is a dilation representation, one cannot expect this construction to be well-behaved---for example, $\op C_m^{\{e\}}(k)=\varnothing$ for $k>1$.
\end{remark}

\begin{remark}
The reader is cautioned that the action of a dilation representation on the plane need not be conformal; in particular, it need not preserve right angles. On the other hand, the action \emph{does} preserve the orthogonality of the standard basis, as Figure 1 suggests, since multiplication on the right by the product of an orthogonal matrix with a diagonal matrix sends the standard basis to an orthogonal basis.
\end{remark}

We now connect these ideas to those of the previous section. Note first that, as an open submanifold of $\mathbb{R}^m$, the manifold $\Box^m$ has a canonical $G$-structure with $\Fr^G_{\Box^m}=\Box^m\times G$ (this observation does not require that $G\to GL(m)$ be a dilation representation). Moreover, a $G$-skew little cube determines a canonical $G$-framed embedding as follows. We have a map $\op C_m^G(k)\to \Map^G(\amalg_k\Fr_{\Box^m}^G,\Fr_{\Box^m}^G)$ defined by sending the $G$-skew little cube $\{(v_i,g_i)\}_{i=1}^k$ to the bundle map $(x,h)\mapsto (f_{v_i,g_i}(x), g_ih)$, where $x$ lies in the $i$th component of $\amalg_k\Box^m$. Since $Df_{v_i,g_i}=\rho(g_i)$, this map, in combination with the projection $\op C_m^G(k)\to \Emb(\amalg_k\Box^m,\Box^m)$, defines the dashed filler in the commuting diagram \[\xymatrix{
\op C_m^G(k)\ar@/^1.4pc/[drr]\ar@/_1.4pc/[ddr]\ar@{-->}[dr]\\
&\Emb^G(\amalg_k\Box^m, \Box^m)\ar[r]\ar[d]& \Map^G(\Fr_{\amalg_k\Box^m}^G, \Fr_{\Box^m}^G)\ar[d]\\
&\Emb(\amalg_k\Box^m, \Box^m)\ar[r]^-D&\Map^{GL(m)}(\Fr_{\amalg_k\Box^m}, \Fr_{\Box^m}),
}\] by composing the induced map to the pullback with the natural map to the homotopy pullback. This map evidently respects composition, so, after fixing a $G$-framed diffeomorphism $\Box^m\cong \mathbb{R}^m$ sending {the origin} to the origin, we obtain a map of operads \[\varphi=\varphi_m^G:\op C_m^G\to \End_{\Mfld_m^G}(\mathbb{R}^m).\]

\begin{theorem}\label{thm:cubes and euclidean}
Let $\rho:G\to GL(m)$ be a dilation representation of a locally compact Hausdorff group. The map $\varphi:\op C_m^G\to \End_{\Mfld_m^G}(\mathbb{R}^m)$ is a weak equivalence.
\end{theorem}

The proof of this result will occupy Section \ref{section:cubes and euclidean proof} below.

\begin{remark}
Since $\Lambda(m)$ is contractible, a dilation framing of $M$ is no more data, up to homotopy, than a trivialization of $TM$. Thus, requiring the structure group $G\to GL(m)$ to be a dilation representation places no serious constraints on the geometry under consideration. The extra flexibility afforded by the dilation group is what permits the definition of the map $\varphi$.
\end{remark}

\begin{remark}
Either using the equivalence of Theorem \ref{thm:cubes and euclidean} as an intermediary, or by direct comparison, one can show that the operad $\op C_m^{SO(m)\cdot\Lambda(m)}$ is weakly equivalent to the so-called ``framed'' little disks operad---see \cite{salvatore-wahl}. More generally, if $G\subseteq O(m)$ is a subgroup, the operad $\op C_m^{G\cdot\Lambda(m)}$ is equivalent to the semidirect product of the little $m$-disks operad with $G$.
\end{remark}

\begin{remark}
The little cubes operads $\op C_m$ are classical \cite{boardman-vogt:lnm, may2}. Versions accommodating the action by a representation of $G$ have been studied by various authors (see \cite{salvatore-wahl}, for example). Typically, one incorporates the action by replacing the cube $\Box^m$ with the closed unit disk $D^m$; however, products of disks are not disks, so these models are not well suited to contexts like ours in which additivity is a key issue---see Section \ref{section:additivity} below. To our knowledge, $\op C_m^G$ is the first little cubes-type model for the homotopy type of the operad $\End_{\Mfld_m^G}(\mathbb{R}^m)$
\end{remark}

\subsection{Additivity}\label{section:additivity}

Fix dilation representations $G\to GL(m)$ and $H\to GL(n)$, and note the canonical identification $\Box^m\times\Box^n=\Box^{m+n}$. If $\{(v_i,g_i)\}_{i=1}^k\in\op C_m^G(k)$ is a $G$-skew little cube, then $\{(v_i,0,\ldots, 0), g_i\}_{i=1}^k$ is a $G\times H$-skew little cube, where $g_i$ is regarded as an element of $G\times H$ via the neutral element of $H$. This assignment is obviously continuous, $\Sigma_k$-equivariant, and respects composition, so we obtain a map of operads $\op C^G_m\to \op C_{m+n}^{G\times H}$. Similarly, we have a map of operads $\op C^H_n\to \op C^{G\times H}_{m+n}$, and these two maps satisfy the interchange relations defining a map $\iota$ from the Boardman-Vogt tensor product.

\begin{conjecture}\label{conj:additivity}
Let $G\to GL(m)$ and $H\to GL(n)$ be dilation representations. The map \[\iota:\op C^G_m\otimes\op C_n^H\to \op C_{m+n}^{G\times H}\] is a weak equivalence.
\end{conjecture}

We defer pursuit of this conjecture to future work, noting only that the classical case of $G=\Lambda(m)$ and $H=\Lambda(n)$ is known to hold by \cite{dunn, brinkmeier}.

We end this section by spelling out the relationship between this additivity map and the comparison of the previous section. Formation of the Cartesian product defines an enriched functor \[\Mfld_m^G\times\Mfld_n^H\to \Mfld_{m+n}^{G\times H},\] which is symmetric monoidal in each variable, since Cartesian products distribute over disjoint unions. This functor gives rise to an operad map \[\End_{\Mfld_m^G}(\mathbb{R}^m)\otimes \End_{\Mfld_n^H}(\mathbb{R}^n)\to \End_{\Mfld_{m+n}^{G\times H}}(\mathbb{R}^{m+n}).\] The following compatibility observation is essentially immediate from the constructions.

\begin{proposition}\label{prop:additivity compatibility}
The diagram of operads \[\xymatrix{
\op C^G_m\otimes\op C^H_n\ar[d]_-{\varphi\otimes\varphi}\ar[r]^-\iota&\op C^{G\times H}_{m+n}\ar[d]^-{\varphi}\\
\End_{\Mfld_m^G}(\mathbb{R}^m)\otimes \End_{\Mfld_n^H}(\mathbb{R}^n)\ar[r]& \End_{\Mfld_{m+n}^{G\times H}}(\mathbb{R}^{m+n})
}\] commutes.
\end{proposition}

\subsection{Proof of Theorem \ref{thm:cubes and euclidean}}\label{section:cubes and euclidean proof}

Define a map $\pi:\op C^G_m(k)\to\Conf_k(\Box^m)\times G^k$ by \[\pi(v_1,g_1,\ldots, v_k,g_k)= (v_1,\ldots, v_k, g_1,\ldots, g_k).\] The proof will be complete upon verifying that $\pi$ is a weak equivalence, for in this case we have the commuting diagram \[\xymatrix{
\op C_m^G(k)\ar[d]_-\pi\ar[r]^-{\varphi(k)}&\op \End_{\Mfld^G_m}(\mathbb{R}^m)(k)\ar[d]^-{\scalebox{.8}{\rotatebox{-90}{$\simeq$}}}\\
\Conf_k(\Box^m)\times G^k\ar[r]^-{\cong}&\Conf_k^G(\mathbb{R}^m),}
\] and the claim follows by two-out-of-three. To show that $\pi$ is a weak equivalence---in fact, a homotopy equivalence---we adapt the argument of \cite[4.8]{may2}. 

We first prove the claim under the assumption that $\rho:G\to GL(m)$ is the inclusion of a subgroup. We begin with the observation that, from the definition of a dilation representation, there is a canonical homeomorphism \begin{align*}G&\xrightarrow{\cong} \big(G\cap O(m)\big)\times\Lambda(m)\\
g&\mapsto (o(g),\lambda(g))
\end{align*} obtained by applying the QR decomposition to the elements of $G$. In particular, $G\simeq G\cap O(m)$.

{Note that the $m$-dimensional box given by the image of the embedding associated to the skew little cube $(v,g)$ has sides of length equal to the eigenvalues of $\lambda(g)$, which will typically not coincide.}

{\begin{definition}
We say that an element $\left\{(v_i,g_i)\right\}_{i=1}^k\in \op C_m^G(k)$ is \emph{equidiameter} if $\lambda(g_i)$ is a scalar matrix that is independent of $i$. We say that $\{(v_i,g_i)\}_{i=1}^k$ is \emph{freewheeling} if it is equidiameter and if $\{(v_i, o_i\lambda(g_i))\}_{i=1}^k$ defines an element of $\op C_m^{O(m)\cdot\Lambda(m)}(k)$ for every $k$-tuple $(o_1,\ldots, o_k)\in O(m)^{\times k}$.
\end{definition}}

{In other words, a skew little cube is equidiameter if the images of its components are hypercubes of equal size, and it is freewheeling if these components may rotate freely in place without colliding or leaving $\Box^m$; equivalently, the corresponding embeddings extend to pairwise disjoint embeddings of the $m$-ball circumscribing $\Box^m$. Thus, an equidiameter element $\{(v_i,g_i)\}_{i=1}^k$ is freewheeling if and only if $\lambda(g_i)\leq \sqrt{2}\min_{1\leq i\neq j\leq k}|v_i-v_j|$.}

{\begin{lemma}\label{lem:freewheeling}
The subspace $\widetilde{\op C}_m^G(k)$ of freewheeling elements in $\op C_m^G(k)$ is a deformation retract.
\end{lemma}
\begin{proof}
First, we deform $\op C_m^G(k)$ onto the subspace of equidiameter elements using the homotopy \[\left\{(v_i,g_i)\right\}_{i=1}^k\mapsto \left\{(v_i, o(g_i)\left((1-t)\lambda(g_i)+t\lambda^{\min})\right)\right\},\] where $\lambda^{\min}$ is the scalar matrix on the the minimal eigenvalue of the $\lambda(g_i)$. One checks easily that this assignment is well-defined, continuous, and fixes the subspace of equidiameter elements pointwise. Second, we further deform this subspace onto the subspace $\widetilde{\op C}_m^G(k)$  of freewheeling elements using the homotopy \[\left\{(v_i,g_i)\right\}_{i=1}^k\mapsto \left\{(v_i, o(g_i)\left((1-t)\lambda(g_i)+t\lambda^{\mathrm{free}})\right)\right\},\] where $\lambda^\mathrm{free}$ is the scalar matrix on the minimum of $\lambda(g_i)$ and $\sqrt{2}\min_{1\leq i\neq j\leq k}|v_i-v_j|$.
\end{proof}}

A similar proof shows that the subspace $\widetilde{\op C}_{m}(k)$ of $\op C_{m}(k)$ consisting of equidiameter classical little cubes such that the corresponding embeddings extend to disjoint embeddings of circumscribed balls is also a deformation retract. From this fact and Lemma \ref{lem:freewheeling}, we now deduce the conclusion of Theorem \ref{thm:cubes and euclidean} in the case of a subgroup by observing the homeomorphism \[\widetilde{\op C}_m^G(k)\xrightarrow{\cong}\widetilde{\op C}_m(k)\times (G\cap O(m))^k\simeq \widetilde{\op C}_m(k)\times G^k.\] Indeed, every freewheeling skew little cube determines a unique equidiameter classical little cube with the same configuration of centers and the same common sidelength as the freewheeling skew little cube. The map shown sends a freewheeling skew little cube to this classical little cube together with the tuple of rotations witnessing it as skew. The inverse map is given by using a $k$-tuple of elements of $G\cap O(m)$ to rotate the components of a freewheeling classical little cube, obtaining thereby a skew little cube. This putative inverse is well-defined by the definition of free-wheeling, and it is clear that each composite is the identity.

We now reduce the general result to this case. Note that, via the homomorphism $\rho$, a $G$-skew little cube determines a $\rho(G)$-skew little cube in the obvious way.

\begin{lemma}
If $G$ is locally compact Hausdorff, then the diagram 
\[\xymatrix{
\op C_m^G(k)\ar[d]\ar[r]&\Conf_k(\Box^m)\times G^k\ar[d]\\
\op C_m^{\rho(G)}(k)\ar[r]&\Conf_k(\Box^m)\times \rho(G)^k
}\] is homotopy Cartesian.
\end{lemma}
\begin{proof}
We claim that the righthand map is a fibration and the diagram is a pullback square. The first claim follows from the general fact that the projection of a locally compact Hausdorff group onto the quotient by a closed subgroup is a fibration \cite[15]{sklyarenko}. The second claim is essentially obvious; for example, the point-set fiber of each vertical map is canonically identified with $\ker\rho^k$.
\end{proof}

The proof is now complete, as the bottom arrow in the above diagram is a weak equivalence by the case of a subgroup; therefore, since the diagram is homotopy Cartesian, the top arrow is a weak equivalence, as claimed.

\section{Proof of the main result}

\subsection{Multi-locality}\label{section:multilocality}

In this section, we introduce the local-to-global technique that we will employ in the proof of the main result. Although we make no use of the machinery of factorization homology as such, our point of view is very much motivated by that theory, and the reader is encouraged to consult \cite{ayala-francis} and \cite[5]{lurie2} for more in this direction.

\begin{definition}
Let $X$ be a topological space and $\mathcal{U}$ a collection of open subsets of $X$. We say that $\mathcal{U}$ is a \emph{Weiss cover} of $X$ if any finite subset of $X$ is contained in some element of $\mathcal{U}$. We say that $\mathcal{U}$ is a \emph{complete Weiss cover} if $\mathcal{U}$ contains a Weiss cover of $\bigcap_{\mathcal{U}_0} U$ for every finite subset $\mathcal{U}_0\subseteq \mathcal{U}$.
\end{definition}

As seen in the lemma below, an important class of complete Weiss covers can be constructed from the following type of cover.

\begin{definition} For any smooth manifold $M$ of dimension $m$, let $\Disk (M)$ denote the collection of open subsets of $M$ diffeomorphic to a disjoint union of copies of $\Bbb R^{m}$.
\end{definition}

\begin{lemma}\label{lem:complete weiss}
Let $M$ and $N$ be manifolds. The collection $\Disk(M)\times\Disk(N)$ is a complete Weiss cover of $M\times N$.
\end{lemma}
\begin{proof}
Given $\{(x_i,y_i)\}_{i=1}^r\subseteq M\times N$, write $S=\{x_i\}_{i=1}^r$ and $T=\{y_i\}_{i=1}^r$. For each $x\in S$ and $y\in T$, choose Euclidean neighborhoods $x\in U_x\subseteq M$ and $y\in V_y\subseteq N$. By shrinking neighborhoods as necessary, we may arrange that each of the collections $\{U_x\}_{x\in S}$ and $\{V_y\}_{y\in T}$ is pairwise disjoint. Then $\coprod_{S\times T}U_x\times V_y$ contains $\{(x_i,y_i)\}_{i=1}^r$, and we conclude that $\Disk(M)\times\Disk(N)$ is a Weiss cover. 

For completeness, suppose we are given $U_j\in \Disk(M)$ and $V_j\in \Disk(N)$ for $1\leq j\leq s$, and set $U=\bigcap_{j=1}^s U_j$ and $V=\bigcap_{j=1}^sV_j$. Then $\bigcap_{j=1}^sU_j\times V_j=U\times V$ has Weiss cover $\Disk(U)\times \Disk(V)\subseteq\Disk(M)\times\Disk(N)$ by the previous argument.
\end{proof}

In particular, taking $N$ to be a singleton, it follows that $\Disk(M)$ is a complete Weiss cover of $M$.

{Since inclusions among open subsets of a $G$-framed manifold $M$ are canonically $G$-framed embeddings among $G$-framed manifolds, a complete Weiss cover $\op U$ gives rise to a functor $\op U\to \Mfld_m^G$, where the source is viewed as a poset and thereby a category. Moreover, for any functor $F:\Mfld_m^G\to \V$, the inclusions into $M$ of the elements of $\op U$ induce a natural transformation from the restriction of $F$ to $\op U$ to the constant functor with value $F(M)$.}

\begin{definition}
Let $\V$ be a model category, $F:\Mfld_m^G\to \V$ a functor, and $M$ a $G$-framed manifold. We say that $F$ is \emph{multi-local} on $M$ if, for any complete Weiss cover $\mathcal{U}$ of $M$, the natural map \[\hocolim_{\mathcal{U}}F\overset{}{\longrightarrow} F(M)\] is an isomorphism in the homotopy category. We say that $F$ is \emph{multi-local} if $F$ is multi-local on $M$ for every $G$-framed manifold $M$.
\end{definition}

Note that, in this definition, we do not require that $\V$ and $F$ be topologically enriched, though they may well be in examples of interest. Indeed, since the criterion for multi-locality involves restricting to the ordinary categories $\op U$, the presence or absence of such an enrichment is inconsequential.

\begin{lemma}\label{lem:conf multi-local}
For each $k\geq0$, the functor $\Conf_k^G:\Mfld_G\to \cat{Top}$ is multi-local.
\end{lemma}
\begin{proof}
If $\op U$ is a complete Weiss cover of $M$, then the collection \[\op U_k:=\{\Conf_k^G(U): U\in \op U\}\] is a complete cover of $\Conf_k^G(U)$ in the sense of \cite[4.5]{dugger-isaksen}, and the claim follows from \cite[4.6]{dugger-isaksen}.
\end{proof}

\subsection{Configuration modules}

In this section, we introduce our model for the configuration spaces of a $G$-framed manifold and prove the main result. First, we outline the essential features that we wish for such a model. We write $\Disk_m^G\subseteq\Mfld_m^G$ for the full subcategory of $G$-framed manifolds that are $G$-framed diffeomorphic to a disjoint union of copies of $\mathbb{R}^m$ with its standard $G$-structure.  Recall the operad map $\iota\colon\op C^G_m\otimes\op C_n^H\to \op C_{m+n}^{G\times H}$ constructed in Section \ref{section:additivity}.

\begin{theorem}\label{thm:cube model}
Let $G\to GL(m)$ and $H\to GL(n)$ be dilation representations of locally compact Hausdorff groups. If Conjecture \ref{conj:additivity} holds for $G$ and $H$, then there is a topologically enriched functor $\op C^G_{(-)}:\cat{Mfld}_m^G\to \Mod{\op C_m^G}(\cat{Top})$ (resp. $n$, $H$) equipped with
\begin{enumerate}
\item a collection of natural $\Sigma_k$-equivariant weak equivalences \[\op C^G_{(-)}(k)\overset{\simeq}{\longrightarrow} \Conf_k^G(-)\] (resp. $H$), and
\item a natural transformation \[\op C^G_{(-)}\otimes \op C^H_{(-)}\to \iota^*(\op C^{G\times H}_{(-\times-)})\] inducing an isomorphism $\op C^G_{U}\otimes^\mathbb{L} \op C^H_{V}\cong \Ho(\iota^*)(\op C^{G\times H}_{U\times V})$ in $\Ho(\Mod{\op C_m^G\otimes\op C_n^H}(\cat {Top}))$ for all $(U,V)\in\Disk_m^G\times\Disk_n^H$.
\end{enumerate}
\end{theorem}

Assuming the existence of functors with these properties, the main result of this article follows easily.

\begin{theorem}\label{thm:fancy main}
Let $G\to GL(m)$ and $H\to GL(n)$ be dilation representations of locally compact Hausdorff groups, $M$ a $G$-framed $m$-manifold, and $N$ an $H$-framed $n$-manifold. If Conjecture \ref{conj:additivity} holds for $G$ and $H$, then there is a natural isomorphism 
\[\Ho( \iota^{*})(\op C_{M\times N}^{G\times H})\cong \op C_M^G\otimes^\mathbb{L} \op C_N^H\] in $ \Ho(\Mod{\op C_m^G\otimes\op C_n^H}(\cat {Top})).$
\end{theorem}

\begin{proof}
The essential point is that the functor $\op C^G_{(-)}:\cat{Mfld}_m^G\to \Mod{\op C_m^G}(\cat{Top})$ is multi-local. Indeed, since homotopy colimits of modules are computed aritywise in $\cat{Top}$ by Lemma \ref{lem:module hocolim}, it suffices to check that $\op C^G_{(-)}(k):\cat{Mfld}_m^G\to\cat{Top}$ is multi-local, which follows from Lemma \ref{lem:conf multi-local} and point (2) of Theorem \ref{thm:cube model}, since multi-locality is invariant under weak equivalence of functors. 

This claim, together with Lemma \ref{lem:complete weiss}, furnishes the first, fifth, and seventh in the following sequence of isomorphisms in $ \Ho(\Mod{\op C_m^G\otimes\op C_n^H}(\cat {Top}))$: 
\begin{align*}
\Ho(\iota^*)(\op C_{M\times N}^{G\times H})&\xleftarrow{\cong}\Ho (\iota^*)\left(\hocolim_{(U,V)\in \Disk(M)\times\Disk(N)}\op C_{U\times V}^{G\times H}\right)\\
&\xleftarrow{\cong}\hocolim_{(U,V)\in \Disk(M)\times\Disk(N)}\Ho(\iota^*)(\op C_{U\times V}^{G\times H})\\
&\xleftarrow{\cong}\hocolim_{(U,V)\in \Disk(M)\times\Disk(N)}\op C_U^G\otimes^\mathbb{L}\op C_V^H \\
&\xrightarrow{\cong}\hocolim_{U\in\Disk(M)}\left(\op C_U^G\otimes^\mathbb{L}\hocolim_{V\in\Disk(N)}\op C_V^H\right)\\
&\xrightarrow{\cong}\hocolim_{U\in\Disk(M)}\op C_U^G\otimes^\mathbb{L} \op C_N^H\\
&\xrightarrow{\cong}\left(\hocolim_{U\in\Disk(M)}\op C_U^G\right)\otimes^\mathbb{L} \op C_N^H\\
&\xrightarrow{\cong}\op C_M^G\otimes^\mathbb{L}\op C_N^H.
\end{align*} The second again uses the fact that restriction along maps of operads preserves homotopy colimits of modules; the third is induced by the map of Theorem \ref{thm:cube model}(2); and the fourth and sixth follow from the observation that the derived Boardman-Vogt tensor product distributes over homotopy colimits in each variable, which is a consequence of Proposition \ref{prop:tensor with cofibrant}.
\end{proof}

Since Conjecture \ref{conj:additivity} is known to hold in the case $G=\Lambda(m)$ and $H=\Lambda(n)$, this completes the proof of Theorem \ref{thm:main}.

\begin{proof}[Proof of Theorem \ref{thm:cube model}]
We define the functor $\op C^G_{(-)}$ to be the composite \[\xymatrix{\op C_{(-)}^G:\Mfld_m^G\ar[rrr]^-{\HHom_{\Mfld_m^G}(\mathbb{R}^m,-)}&&&\Mod{\End_{\Mfld_m^G}(\mathbb{R}^m)}(\cat{Top})\ar[r]^-{\varphi^*}&\Mod{\op C_m^G}(\cat{Top}),}\] where the first arrow is the restricted Yoneda embedding of Example \ref{example:yoneda}. For point (1), we have the equivalence
\[\op C_{M}^G(k)= \Emb^G(\amalg_k\mathbb{R}^m,M)\overset{\sim}{\longrightarrow}\Conf_k^G(M)\] of Proposition \ref{prop:embedding and conf}.

To define the natural transformation of point (2), we note that, by the commuting diagram \[\xymatrix{
\op C_m^G\otimes\op C_n^H\ar[d]_-{\varphi\otimes\varphi}\ar[r]^-\iota&\op C_{m+n}^{G\times H}\ar[d]^-{\varphi}\\
\End_{\Mfld_m^G}(\mathbb{R}^m)\otimes\End_{\Mfld_n^H}(\mathbb{R}^n)\ar[r]^-\jmath&\End_{\Mfld_{m+n}^{G\times H}}(\mathbb{R}^{m+n})
}\] of Proposition \ref{prop:additivity compatibility}, it suffices to exhibit a natural transformation \[\HHom_{\Mfld_m^G}(\mathbb{R}^m,-)\otimes \HHom_{\Mfld_n^H}(\mathbb{R}^n,-)\to \jmath^*\HHom_{\Mfld_{m+n}^{G\times H}}(\mathbb{R}^{m+n},-).\] In order to accomplish this task, we will appeal to the commuting diagram \[\xymatrix{
\cat F(\End_{\Mfld_m^G}(\mathbb{R}^m))\times\cat F(\End_{\Mfld_n^H}(\mathbb{R}^n))\ar[d]_-{\mu}\ar[r]&\Disk_m^G\times\Disk_n^H\ar[r]\ar[dd]^-{\Pi}&\Mfld_m^G\times\Mfld_n^H\ar[dd]^-{\Pi}\\ 
\cat F(\End_{\Mfld_m^G}(\mathbb{R}^m)\otimes\End_{\Mfld_n^H}(\mathbb{R}^n))\ar[d]_-{\cat F(\jmath)}\\
\cat F(\End_{\Mfld_{m+n}^{G\times H}}(\mathbb{R}^{m+n}))\ar[r]&\Disk_{m+n}^{G\times H}\ar[r]&\Mfld_{m+n}^{G\times H},
}\] where $\mu$ is the enriched functor of Section \ref{sec:opmod}, and $\Pi$ is the enriched functor sending a $G$-framed manifold $M$ and an $H$-framed manifold $N$ to the product $M\times N$ with its canonical $G\times H$-framing (see Section \ref{section:modules as presheaves} for a reminder of the definition of $\cat F$).

Now, by definition, the lifted Boardman-Vogt tensor product is the left Kan extension of the functor $\HHom_{\Mfld_m^G}(\mathbb{R}^m,-)\boxtimes \HHom_{\Mfld_n^H}(\mathbb{R}^n,-)$ along the functor $\mu$. Thus, by the universal property of the left adjoint, exhibiting the desired natural transformation is equivalent to giving a map from $\HHom_{\Mfld_m^G}(\mathbb{R}^m,-)\boxtimes \HHom_{\Mfld_n^H}(\mathbb{R}^n,-)$ to the restriction of $\HHom_{\Mfld_{m+n}^{G\times H}}(\mathbb{R}^{m+n},-)$ along the composite $\cat F(\jmath)\mu$. Since the source of the putative natural transformation is obtained by restricting $\Emb^G(-,M)\times\Emb^H(-,N)$ along the top horizontal composite, while the target is obtained by restricting $\Emb^{G\times H}(-, M\times N)$ along the full counterclockwise composite, the natural inclusion of functors $\Emb^G(-,M)\times\Emb^H(-,N)\hookrightarrow \Emb^{G\times H}(-, M\times N)\circ \Pi$ supplies the desired natural transformation.

To see that this natural transformation induces an isomorphism $\op C^G_{U}\otimes^\mathbb{L} \op C^H_{V}\cong \Ho(\iota^*)(\op C^{G\times H}_{U\times V})$ in $\Ho(\Mod{\op C_m^G\otimes\op C_n^H}(\cat {Top}))$ for all $(U,V)\in\Disk_m^G\times\Disk_n^H$, it suffices to supply weak equivalences $\Sigma_k\circ \op C_m^G\xrightarrow{\simeq} \op C^G_{\amalg_k\mathbb{R}^m}$, where $\Sigma_k$ is considered as a symmetric sequence concentrated in arity $k$ (resp. $(\ell,n, H)$, $(k+\ell, m+n, G\times H)$), and to fit these maps into a commuting diagram 
\begin{equation}\label{diagram}
\xymatrix{
(\Sigma_k\circ\op C_m^G)\otimes (\Sigma_\ell\circ \op C_n^H)\ar[d]\ar[r]&\iota^*(\Sigma_{k+\ell}\circ\op C_{m+n}^{G\times H})\ar[d]\\
\op C_{\amalg_k\mathbb{R}^m}^G\otimes \op C_{\amalg_\ell\mathbb{R}^n}^H\ar[r]&\iota^*(\op C_{\amalg_{k+\ell}\mathbb{R}^{m+n}}^{G\times H})
}
\end{equation}
of $\op C_m^G\op \otimes \op C_n^H$-modules. Indeed, the top arrow is a weak equivalence by the calculation \begin{align*}
(\Sigma_k\circ\op C_m^G)\otimes (\Sigma_\ell\circ \op C_n^H)&\cong (\Sigma_k\Box\Sigma_\ell)\circ\op C_m^G\otimes\op E_n^H\\
&\cong \Sigma_{k+\ell}\circ\op C_m^G\otimes\op C_n^H\\
&\xrightarrow{\simeq}\iota^*(\Sigma_{k+\ell}\circ\op C_{m+n}^{G\times H}),
\end{align*} where $\Box$ denotes the matrix tensor product (see Example \ref{ex:matrix-mon}). Here we have used Conjecture \ref{conj:additivity} in the last step, and the upper lefthand expression computes $\op C^G_{\amalg_k\mathbb{R}^m}\otimes^\mathbb{L}\op C^H_{\amalg_\ell\mathbb{R}^n}$, since $\Sigma_k\circ\op C_m^G$ is a free, hence cofibrant, $\op C_m^G$-module (resp. $\ell$, $n$, $H$).

The desired weak equivalence $\Sigma_k\circ \op C_m^G\to \op C^G_{\amalg_k\mathbb{R}^m}$ is supplied by the universal property of the free module and the $\Sigma_k$-equivariant map \[\Sigma_k\to \Emb^G(\amalg_k\mathbb{R}^m,\amalg_k\mathbb{R}^m)\] defined by sending a permutation $\sigma$ to the unique embedding that is the identity componentwise and induces $\sigma$ on connected components (resp. $(\ell, n, H)$, $(k+\ell, m+n, G\times H)$). To see that this map is an equivalence, we appeal to the calculation \begin{align*}
\op C^G_{\amalg_k\mathbb{R}^m}(r)&=\Emb^G(\amalg_r\mathbb{R}^m, \amalg_k\mathbb{R}^m)\\
&\cong \coprod_{r_1+\cdots +r_k=r}\prod_{i=1}^k\Emb^G(\amalg_{r_i}\mathbb{R}^m,\mathbb{R}^m)\times_{\prod_{i=1}^k\Sigma_{r_i}}\Sigma_r\\
&\cong\Emb^G(-,\mathbb{R}^m)^{\odot k}(r)\\
&\xleftarrow{\simeq} (\Sigma_1\circ \op C_m^G)^{\odot k}(r)\\
&\cong(\Sigma_1^{\odot k}\circ\op C_m^G)(r)\\
&\cong (\Sigma_k\circ\op C_m^G)(r)
\end{align*} (recall that $\odot$ denotes the graded tensor product of symmetric sequences defined in Section \ref{sec:operads and modules}). Finally, after unwinding universal properties, diagram (\ref{diagram}) commutes because
\[\xymatrix{
\Sigma_k\times\Sigma_\ell\ar[d]\ar[r]&\Sigma_{k+\ell}\ar[d]\\
\Emb^G(\amalg_k\mathbb{R}^m,\amalg_k\mathbb{R}^m)\times\Emb^H(\amalg_\ell\mathbb{R}^m,\amalg_\ell\mathbb{R}^m)\ar[r]&\Emb^{G\times H}(\amalg_{k+\ell}\mathbb{R}^m, \amalg_{k+\ell}\mathbb{R}^m)
}\] does. 
\end{proof}

\bibliographystyle{amsalpha} 
\bibliography{raao.bib}

\end{document}